\documentclass[a4paper, 12pt, oneside]{amsart}
\usepackage[T1]{fontenc}              
\usepackage[latin1]{inputenc}         
\usepackage{geometry}  
\usepackage{bbm}               

\usepackage{amsmath,amsfonts,amssymb,amsrefs} 
\usepackage{amsthm}                   
\usepackage{graphicx}                 
\usepackage{subfig}                   
\usepackage{listings}                 
\usepackage{enumerate}
\usepackage{tikz-cd}
\usepackage{tikz}
\usepackage{rotating}
\usepackage[capitalize]{cleveref}
\usepackage{soul}

\newtheorem{lem}{Lemma}[section]
\newtheorem{prop}[lem]{Proposition}
\newtheorem{cor}[lem]{Corollary}
\newtheorem{thm}[lem]{Theorem}

\theoremstyle{definition}

\newtheorem{remark}[lem]{Remark}

\DeclareMathOperator{\rad}{rad}

\DeclareMathOperator{\soc}{soc}
\DeclareMathOperator{\tr}{tr}

\DeclareMathOperator{\Ext}{Ext}

\DeclareMathOperator{\Hom}{Hom}

\DeclareMathOperator{\modu}{mod}

\DeclareMathOperator{\HH}{HH}

\DeclareMathOperator{\Out}{Out}
\DeclareMathOperator{\Aut}{Aut}
\DeclareMathOperator{\op}{op}

\DeclareMathOperator{\Z}{Z}

\makeatletter
\@namedef{subjclassname@2020}{\textup{2020} Mathematics Subject Classification}
\makeatother

\begin{document}

\title[\resizebox{5.25in}{!}{Skew Group Algebras, (\textbf{Fg}) and Self-injective Rad-Cube-Zero Algebras}]{Skew Group Algebras, the (\textbf{Fg}) Property and Self-injective Radical Cube Zero Algebras}
\author[M. H. Sand{\o}y]{Mads Hustad Sand{\o}y}
\address{Department of mathematical sciences, NTNU, Trondheim, Norway}
\email{mads.sandoy@ntnu.no}

\begin{abstract}
We classify self-injective radical cube zero algebras with respect to whether they satisfy certain finite generation conditions sufficient to have a fruitful theory of support varieties defined via Hochschild cohomology in the vein of \cites{Erdmann-et-al, SS04}. 
Using skew group algebras and Linckelmann's notion of separable equivalence, we obtain results that complement the existing partial classification of \cite{Said'15} and complete the classification begun in \cites{Said'15, Erdmann-Solberg-2} up to assumptions on the characteristic of the field. 
\end{abstract}

\maketitle

\section{Introduction}
Support varieties of modules are powerful tools, but they are not always available. 
In the case of a group algebra of a finite group, one defines them using the maximal ideal spectrum of the group cohomology. 
One generalization of this notion to a more general finite dimensional algebra $\Lambda$ works instead with a subalgebra of the Hochschild cohomology ring of $\Lambda$. 
In \cite{Erdmann-et-al}, many of the results one would expect a good notion of support to satisfy were shown to hold assuming $\Lambda$ is self-injective and has certain finite generation properties defined as follows: One says that $\Lambda$ has (\textbf{Fg})
provided 
$$\Ext^{*}_{\Lambda^e}(\Lambda, U) = \oplus_{i \geq 0} \Ext^{i}_{\Lambda^e}(\Lambda, U)$$
is a Noetherian module over the Hochschild cohomology ring of $\Lambda$
$$\HH^*(\Lambda) = \Ext^*_{\Lambda^{e}}(\Lambda, \Lambda)$$ for every finitely generated $\Lambda^{e}$-module $U$, where $\Lambda^{e} :=   \Lambda^{\op}\otimes_k \Lambda$ is the enveloping algebra of $\Lambda$.\footnote{While this is not the original definition given in \cite{Erdmann-et-al}*{Assumption 1, Proposition 1.4, Assumption 2}, it follows by  \cite{Solberg-survey}*{Propositions 5.5-5.7} that we can assume without loss of generality that we work with $H = \HH^*(\Lambda)$.} 

A natural question is then whether particular classes of algebras possess this finite generation property. 
In \cite{Erdmann-Solberg-1}, this was settled for the case of radical-cube-zero weakly symmetric algebras. 
Moreover, as the (\textbf{Fg}) property implies finite complexity, \cite{Erdmann-Solberg-2} described the quivers and relations of representation infinite radical-cube-zero self-injective algebras $\Lambda$ of finite complexity as a step towards determining which of these satisfied the aforementioned finite generation conditions. 
In particular, they used the type of these algebras, which is defined as the extended Dynkin type of the separated quiver of the radical-square-zero algebra obtained by taking such an algebra $\Lambda$ modulo its socle. 

In the thesis \cite{Said'15}, this work was essentially completed for classes of type $\widetilde{A}_n$ and $\widetilde{D}_n$, but not for the exceptional types. 
Many of the proofs involve long explicit computations using quivers with relations of coverings of the algebras in question.

In this note, we show that the general radical-cube-zero self-injective case can be reduced to the weakly-symmetric one by using skew group algebras and the notion of separable equivalence as introduced in \cite{Linckelmann'11} provided the characteristic is good and the Nakayama automorphism of the algebra has finite order. 
Note that the latter always holds if the type of the algebra is not $\widetilde{A}_n$. 
Using this, we partially recover the main result of \cite{Said'15}: Namely, we do so in part for type $\widetilde{A}_n$ and in whole for type $\widetilde{D}_n$. 
However, we also cover the exceptional cases, something \cite{Said'15} does not.
Hence, combining our results with those those of \cite{Said'15} we obtain a complete answer to when a radical-cube-zero self-injective algebra satisfies the (\textbf{Fg}) property. 

\section*{Acknowledgments}
The author would like to thank {\O}yvind Solberg for helpful discussions, and for careful reading of and helpful suggestions on a previous version of this paper. 
The author is also grateful to have been supported by Norwegian Research Council project 301375, ``Applications of reduction techniques and computations in representation theory'' during later revisions of the paper.
\section*{Acknowledgments}

\section{Skew group algebras, the (\textbf{Fg}) property and self-injective radical-cube-zero algebras}
For the rest of this note, let $k$ be an algebraically closed field and $\Lambda$ a finite dimensional $k$-algebra. 
We often write $\otimes$ instead of $\otimes_k$. 
Moreover, we denote the $k$-duality functor by $D$.
We refer to \cites{ARS97, ASS06} for background on quivers with relations and the representation theory of finite dimensional algebras. 

Let $G$ be a finite group acting on $\Lambda$, by which we mean that there exists some group homomorphism $G \rightarrow \operatorname{Aut} \Lambda$ with $\operatorname{Aut} \Lambda$ the multiplicative group of algebra automorphisms of $\Lambda$. 
The \textit{skew group algebra} $\Lambda G$ is then the finite dimensional algebra with underlying vector space $\Lambda \otimes_k kG$, in which $kG$ is the group algebra of $G$, and with multiplication defined as 
\[(\lambda \otimes g)(\lambda' \otimes g') :=  \lambda g(\lambda') \otimes gg'. \]
Observe also that there is an isomorphism 
\[\Lambda G \simeq \bigoplus_{g \in G} {}_{1}\Lambda_g \]
obtainable by  
using the fact that every element can be written uniquely as a sum $\sum_{g \in G} \lambda_g \otimes g$.
Recall that ${}_{1}\Lambda_g$ is the bimodule with underlying vector space $\Lambda$ and with bimodule action given by twisting on the right with $g$, i.e.\ 
\[\lambda \cdot \lambda' \cdot \lambda'' = \lambda \lambda' g(\lambda''). \]
Consequently, the skew group algebra $\Lambda G$ is a free left $\Lambda$-module. 

We now relate this to Linckelmann's notion of separable equivalence of algebras \cite{Linckelmann'11}. 
Note thus that the term ${}_{1}\Lambda_{e}$ in $\oplus_{g \in G} \,  {}_{1}\Lambda_g$ is equal to $\Lambda$. 
Hence, if we let $M$ and $N$ denote respectively $\Lambda G$ considered as a $\Lambda$-$\Lambda G$-bimodule and as a $\Lambda G$-$\Lambda$-bimodule, we see that $\Lambda$ is isomorphic to a direct summand of $M \otimes_{\Lambda G} N$ as a $\Lambda$-$\Lambda$-bimodule. 
On the other hand, provided the order of $G$ is invertible in $k$, it follows by \cite{Reiten-Riedtmann}*{Theorem 1.1} that $N \otimes_\Lambda M$ has a $\Lambda G$-$\Lambda G$-bimodule direct summand isomorphic to $\Lambda G$. Following \cite{Linckelmann'11}, one thus calls $\Lambda$ and $\Lambda G$ \textit{separably equivalent}. 

The following result, the proof of which is essentially \cite{Linckelmann'11}*{Theorem 4.1}, allows us to transfer the finite generation properties between $\Lambda$ and $\Lambda G$. The proof we give indicates which parts of the proof of \cite{Linckelmann'11}*{Theorem 4.1} that must be altered to accommodate our slightly different assumptions. In particular, we indicate how we can circumvent the assumption that the algebras in question must be symmetric. Hence, recall that an algebra $\Lambda$ is \textit{symmetric} if $\Lambda$ and $D(\Lambda)$ are isomorphic as bimodules. 

\begin{prop}\label{skew-group-transfer-prop}
Assume that $\Lambda$ is a $k$-algebra, that the order of $G$ is invertible in $k$, and let $A = \Lambda G$. Then $\Ext_{\Lambda^{e}}(\Lambda,U)$ is Noetherian as an $\HH^*(\Lambda)$-module for any finitely generated $\Lambda^{e}$-module $U$ if and only if $\Ext_{A^{e}}(A,V)$ is Noetherian as an $\HH^*(A)$-module for any finitely generated $A^{e}$-module $V$. In other words, $\Lambda$ is \emph{(\textbf{Fg})} if and only if $A = \Lambda G$ is.
\end{prop} 

\begin{proof}
As mentioned, the proof is essentially the same as that of \cite{Linckelmann'11}*{Theorem 4.1} albeit with our $A$-$\Lambda$-bimodule $N$ substituted for $D(M)$.

In particular, note that \cite{Linckelmann'11}*{Theorem 4.1} requires symmetric to use two lemmas: 
\begin{enumerate}[(i)]
    \item \cite{Linckelmann'11}*{Lemma 3.2}; and
    \item \cite{Linckelmann'11}*{Lemma 4.5}.
\end{enumerate}

For (ii), note that the parts of the proof of \cite{Linckelmann'11}*{Lemma 4.5} which use the symmetric assumption are (1): The proof invoking \cite{Linckelmann'11}*{Lemma 4.4}, the content of which is that the symmetric assumption forces the $\Lambda$-$\Lambda G$-bimodule $M$ to have $D(M)$ as both its left and right adjoint, and that this then implies that $D(M) \otimes_\Lambda - \otimes_{\Lambda} M$ as a functor from $\modu \Lambda^{e}$ to $\modu A^{e}$ is left adjoint to the functor $M \otimes_{A} - \otimes_{A} D(M)$; and (2): that symmetric implies that $D(M)$ is right projective since $M$ is left projective.
For (i), note that the conclusion of \cite{Linckelmann'11}*{Lemma 3.2} essentially says that $\Lambda$ and $A$ are respectively bimodule summands of $M \otimes_{A} D(M)$ and $D(M) \otimes_{\Lambda} M$.

Observe then that \cite{Reiten-Riedtmann}*{Theorem 1.1} shows $- \otimes_{\Lambda G} N \simeq \Hom_{\Lambda G}(N, -)$ to be left adjoint to $- \otimes_{\Lambda} M$ in addition to it being its right adjoint by $\otimes$-$\Hom$-adjunction. Consequently, we deduce that $N \otimes_{\Lambda} - \otimes_{\Lambda} M$ is left adjoint to $M \otimes_{A} - \otimes_{A} N$. 
Hence, (ii)(1) is taken care of and (ii)(2) is no issue as the $\Lambda G$-$\Lambda$-bimodule $N = \Lambda G$ is right projective. Consequently, \cite{Linckelmann'11}*{Lemma 4.5} holds with $N$ substituted for $D(M)$.

Moreover, $M$ and $N$ both satisfy the conclusion in the lemma in (i) by the discussion before the proposition. 
\end{proof} 

\begin{remark} 
Note that \cite[Theorem 4.2(2)]{Bergh'23} is a similar result that was shown independently, as is also pointed out in \cite{Bergh'23}.

Also note that the only part of the proposition above that requires any assumption on the field $k$ is the conclusion that $\Lambda$ is (\textbf{Fg}) if and only if $A = \Lambda G$ is (\textbf{Fg}), and for this assuming that $k$ is perfect instead of algbraically closed would suffice; see also \cite[Remark 3.3]{Bergh'23} for a more thorough discussion. 
\end{remark}

Recall now that if $\Lambda$ is a basic self-injective $k$-algebra, we can assume that $\Lambda$ is a Frobenius algebra, implying that there is a $\Lambda^{e}$-module isomorphism $D(\Lambda) \simeq {}_{1}\Lambda_{\nu}$ where $\nu$ is an algebra automorphism of $\Lambda$. 
Here, $\nu$ is what is called the \textit{Nakayama automorphism} of $\Lambda$. 
Recall that $\nu$ is unique up to composition with an inner automorphism, i.e.\ an automorphism induced by conjugation with an invertible element of $\Lambda$.
Note that $\Lambda$ is symmetric if and only if it is Frobenius and its Nakayama automorphism is an inner automorphism. 

It is well known that for an algebra $\Lambda$ to be Frobenius is equivalent to there existing a non-degenerate associative bilinear form $\langle - , - \rangle \colon \Lambda \times \Lambda \rightarrow k$, in which associative means that $\langle \lambda, \lambda' \lambda'' \rangle = \langle \lambda \lambda', \lambda'' \rangle$. For example, the group algebra $kG$ is Frobenius with the bilinear form defined on the basis of group elements $g, h \in G$ by $\langle g, h \rangle = 1$ if $gh = e$ and $0$ otherwise. 

Using this, one can show that if $\Lambda$ is Frobenius then $\Lambda G$ is as well: One defines a bilinear form on $\Lambda G$ by letting
\[\langle \lambda \otimes g, \lambda' \otimes g' \rangle_{\Lambda G} :=  \langle \lambda, g(\lambda') \rangle_{\Lambda} \cdot \langle g, g'\rangle_{kG}.\]
Note that we here use subscripts to denote which algebra a given form is defined for, but often omit these in the sequel as it will cause no ambiguity. 

It is straightforward to see that this form is associative, and so we need only check non-degeneracy: Assume thus that $x = \sum_{g \in G} \lambda_g \otimes g$ satisfies $\langle x, y \rangle = 0$ for all $y \in \Lambda G$. 
Hence, we deduce that
\begin{align*}
\langle x, \lambda'\otimes h^{-1} \rangle 
& = \langle \lambda_{h}, h(\lambda') \rangle \cdot \langle h, h^{-1} \rangle \\
& = \langle \lambda_{h}, h(\lambda') \rangle \\
& = 0
\end{align*}
for all all $\lambda' \in \Lambda$. But $h$ acts on elements of $\Lambda$ by its image in $\operatorname{Aut} \Lambda$, i.e.\ via an algebra automorphism, and so we get that $\langle x, \Lambda \rangle = 0$, a contradiction. 

We also need the following convenient lemma.

\begin{lem}\label{Holm-Zim-lem}\cite{Holm-Zimmermann}*{Lemma 2.7}
Let $k$ be a field and let $\Lambda$ be a finite dimensional Frobenius $k$-algebra with associated bilinear form $\langle -,- \rangle$. Then the Nakayama automorphism $\nu$ of $\Lambda$, satisfies $\langle \lambda, \lambda'\rangle = \langle \lambda', \nu(\lambda) \rangle$ for all $\lambda,\lambda' \in \Lambda$, and any automorphism satisfying this formula is a Nakayama automorphism.
\end{lem}

The next result determines the Nakayama automorphism of $\Lambda G$ whenever $\Lambda$ has a Nakayama automorphism of finite order and $G$ is generated by $\nu$.

\begin{prop}\label{skew-group-is-sym-lem}
Let $\Lambda$ be a basic self-injective algebra with Nakayama automorphism $\nu$ of finite order, and let $G$ be the finite cyclic group generated by $\nu$. Then $\Lambda G$ is symmetric. 
\end{prop}
\begin{proof}
We compute as follows
\begin{align*}
\langle \lambda \otimes \nu^{i}, \lambda'\otimes \nu^{-i} \rangle 
& = \langle \lambda, \nu^{i}(\lambda') \rangle \cdot \langle \nu^{i}, \nu^{-i} \rangle \\
& = \langle \lambda, \nu^{i}(\lambda') \rangle \\
& = \langle \nu^{i}(\lambda'), \nu(\lambda) \rangle \\
& = \langle \lambda',\nu^{-i + 1}(\lambda) \rangle \\
& = \langle \lambda' \otimes  \nu^{-i},  \nu(\lambda) \otimes \nu^{i} \rangle.
\end{align*}
This shows that 
$$\nu_{\Lambda G}(\lambda \otimes \nu^{i}) = \nu(\lambda) \otimes \nu^{i}$$ 
is a Nakayama automorphism of $\Lambda G$ by \cref{Holm-Zim-lem} as long as it is an automorphism of the algebra. 

Observe now that $1 \otimes \nu \in \Lambda G$ is invertible, and that it defines an inner automorphism 
\begin{align*}
\nu(\lambda) \otimes \nu^{i} & 
= (1 \otimes \nu) \cdot (\lambda \otimes \nu^{i}) \cdot (1 \otimes \nu^{-1}).
\end{align*} 
We are hence done as this implies that we can choose $\Lambda G$ to have the identity as its Nakayama automorphism.
\end{proof}

We also need the following result.

\begin{prop}\label{basic-version-also-sym-prop} \cite{Giovannini-Pasquali}*{Lemma 2.7}
If $\Lambda$ is a symmetric algebra and $\eta$ is some idempotent of $\Lambda$, then $\eta \Lambda \eta$ is symmetric. 
\end{prop}

We are not yet ready to show the main (and only!) result of this note, namely that we can reduce the classification of basic radical-cube-zero self-injective (\textbf{Fg}) algebras to the classification in the weakly symmetric case in all cases except those of type $\widetilde{A}_n$ provided the characteristic is good. 
In particular, we need to show that all self-injective radical-cube-zero algebras that are not of type $\widetilde{A}_n$ have a Nakayama automorphism that is of finite order. 

Note that an algebra $\Lambda$ is called \textit{radical-cube-zero} if it satisfies $\rad^3 \Lambda = 0$. 
Moreover, we can ignore those algebras which satisfy $\rad^2 \Lambda = 0$ in the basic self-injective case which is of interest to us for reasons we now explain. 
Indeed, we can restrict to those algebras $\Lambda$ which are representation-infinite since the ones which are representation-finite are already known to be (\textbf{Fg}) essentially by the main result in \cite{Dugas'10} via the methods in \cite{Green-et-al-2}. 
If $\rad^2 \Lambda = 0$ holds with $\Lambda$ basic self-injective, then one can check that $\Lambda$ must be a Nakayama algebra, and hence of finite representation type; see e.g.\ \cite{Erdmann-Solberg-1}*{Proposition 1.4}.

By \cite{Martinez-Villa-survey}, any representation-infinite self-injective radical-cube-zero algebra is Koszul; see \cite{Beilinson-Ginzburg-Soergel} or \cite{Martinez-Villa-survey} for a definition. 
In particular, this implies that it must be a quadratic algebra and hence have quadratic relations if it is basic. 
Moreover, note that a Koszul algebra $\Lambda$ is a \textit{positively graded algebra}, meaning that as a $k$-vector space $\Lambda$ can be expressed as a direct sum of the form $\oplus_{i \in \mathbb{N}} \Lambda_i$ satisfying that $\Lambda_i \cdot \Lambda_j \subseteq \Lambda_{i + j}$.
Also note that a finite dimensional Koszul algebra is isomorphic as a graded algebra to its associated graded with respect to the radical filtration by \cite{Beilinson-Ginzburg-Soergel}*{Proposition 2.5.1}.

We want to combine this with some information on the structure of the quivers of self-injective radical-cube-zero algebras.
Following \cite{Erdmann-Solberg-2}*{Definition 7.1}, the \textit{type} of such an algebra $\Lambda$ is the underlying graph of the separated quiver of $A = \Lambda/\soc \Lambda$, which is a radical-square-zero algebra. Following \cite{ARS97}, the separated quiver of a radical-square-zero algebra $A$ is given by the quiver of the bipartite hereditary algebra
\[
\begin{bmatrix}
A/\rad A & \rad A \\
0 & A/\rad A
\end{bmatrix}.
\]
Using the aforementioned \cite{Beilinson-Ginzburg-Soergel}*{Proposition 2.5.1}, we get that this is isomorphic to
\[
\nabla \Lambda :=  
\begin{bmatrix}
\Lambda_0 & \Lambda_1 \\
0 & \Lambda_0
\end{bmatrix},
\]
i.e.\ the Beilinson algebra of $\Lambda$ as in \cites{Yama13, MM11}.

Recall now that the $2$\textit{-quasi-Veronese} of a graded algebra $\Lambda$ (as in \cite{MM11}*{Definition 4.6.2}) of highest degree $2$ is 
\[
\Lambda^{[2]} :=  
\begin{bmatrix}
\Lambda_0 & \Lambda_1 \\
0 & \Lambda_0
\end{bmatrix}
\oplus
\begin{bmatrix}
\Lambda_2 & 0 \\
\Lambda_1 & \Lambda_2
\end{bmatrix}.
\]

The essential idea in the following is that we can replace $\Lambda$ by its $2$-quasi-Veronese $\Lambda^{[2]}$ to get what we can think of as a normal form of $\Lambda$ which is easier to work with.

Recall now that given a graded module $M= \oplus_{i\in\Z}M_i$, one defines the \textit{$j$-th graded shift} of $M$ to be the graded module $M\langle j \rangle$ with $M\langle j \rangle_i = M_{i-j}$.
Since $\Lambda$ is Frobenius and its socle lies in its highest degree, i.e.\ 2, we have that $D\Lambda \langle -2 \rangle \simeq {}_{1}\Lambda_{\nu}$ as graded bimodules by \cite{Haugland-Sandoy}*{Lemma 2.2}. 
Consequently, we get $\Lambda_2 \simeq {}_{1}D(\Lambda_0)_{\nu^{-1}}$ and $\Lambda_1 \simeq {}_{1}D(\Lambda_1)_{\nu^{-1}}$.
Note that we here abuse notation by letting $\nu^{-1}$ denote both the inverse of the Nakayama automorphism of $\Lambda$ and the automorphism the inverse induces on $\Lambda^{[2]}$. 
Using this, we deduce that 
\[
\Lambda^{[2]} \simeq 
\begin{bmatrix}
\Lambda_0 & \Lambda_1 \\
0 & \Lambda_0
\end{bmatrix}
\oplus
\begin{bmatrix}
{}_{1}D(\Lambda_0)_{\nu^{-1}}& 0 \\
{}_{1}D(\Lambda_1)_{\nu^{-1}} & {}_{1}D(\Lambda_0)_{\nu^{-1}}
\end{bmatrix}.
\]

In other words, $\Lambda^{[2]}$ is the twisted trivial extension of $(\Lambda^{[2]})_0$ with respect to the degree $0$ part of $\nu^{-1}$, a notion we recall now: 
Given a finite dimensional algebra $A$, the \textit{trivial extension} of $A$ is $\Delta A :=  A \oplus DA$ as a vector space. 
It is an algebra with multiplication $(a, f) \cdot (b, g) = (ab, ag + fb)$ for $a,b \in A$ and $f,g \in DA$. 
Observe that $\Delta A$ is a symmetric algebra with a non-degenerate bilinear form induced by setting $\langle a, f \rangle 
:=  (a\cdot f)(1)$ and $\langle f, a \rangle 
:=  (f\cdot a)(1)$ for $a \in A$ and $f \in DA$
and extending bilinearly. 

Similarly, if $\sigma$ is an algebra automorphism of $A$, one defines the \textit{twisted trivial extension} of $A$ \textit{with respect to} $\sigma$ as the vector space $\Delta_{\sigma} A :=  A \oplus {}_{1}D(A)_{\sigma}$ endowed with the same multiplication as for the usual trivial extension. 
Note, however, that $\Delta_{\sigma} A$ is not necessarily symmetric, only Frobenius. 
This is well known, but we include a proof for the convenience of the reader:

If $A$ is some algebra and $\sigma$ is an automorphism of $A$, then the twisted trivial extension of $A$ with respect to $\sigma$ has Nakayama automorphism given essentially by $\sigma^{-1}$, and we abuse notation and let it be denoted by $\sigma^{-1}$, too. 
To see this, let $\sigma^{-1}$ act on $a \in A$ as you would expect while $\sigma^{-1}$ acts on $f \in DA$ by $\sigma^{-1}(f) = f \circ \sigma$. 
One can check that setting $\langle a, f \rangle 
:=  (a\cdot f)(1)$ and $\langle f, a \rangle 
:=  (f\cdot a)(1)$ for $a \in A$ and $f \in DA$ induces a non-degenerate associative bilinear form for $\Delta_{\sigma} A$. 

Note then that 
\[
\langle a, f \rangle 
= (a\cdot f)(1)
= f(a)
= (f \cdot \sigma^{-1}(a))(1) 
= \langle f, \sigma^{-1}(a) \rangle
\]
and 
\[
\langle f, a \rangle 
= (f\cdot a)(1)
= f(\sigma(a))
= (a \cdot f \circ \sigma)(1)  
= \langle a, \sigma^{-1}(f) \rangle,
\]
so the claim follows by \cref{Holm-Zim-lem}, and we have hence shown: 

\begin{prop}
The twisted trivial extension of $A$ with respect to the automorphism $\sigma$ has Nakayama automorphism given by $\sigma^{-1}$, where $\sigma^{-1}$ acts on elements of $A$ as one would expect and $\sigma^{-1}(f) = f \circ \sigma$ for $f \in D(A)_{\sigma}$.
\end{prop}

Consequently, we see that the possible forms of the Nakayama automorphisms of a $2$-quasi-Veronese $\Lambda^{[2]}$ should be determined by the automorphisms of the algebra $\Lambda^{[2]}_0$.
Also note that $\Lambda^{[2]}_0$ is a bipartite hereditary algebra whenever $\Lambda$ is basic radical-cube-zero self-injective, and that the type of $\Lambda^{[2]}_0$ coincides with the type of $\Lambda$.
Our aim to is to use this to somehow exploit the fact that tame hereditary algebras have particularly nice groups of outer automorphisms, and especially so if they are not of type $\widetilde{A}_n$.

To achieve this, we first show that when checking if $\Lambda$ is (\textbf{Fg}) we can reduce to working with the $2$-quasi-Veronese of $\Lambda$ provided the characteristic is not $2$. 
To do this, we begin by noting that $\Lambda^{[2]}$ can be identified with a smash product of $\Lambda$ with the group $\mathbb{Z}_2$, a notion we now recall: First, recall  that if $G$ is a group, the notion of a $G$-\textit{graded algebra} $\Lambda$ is defined similarly to that of a positively graded algebra, i.e.\ $\Lambda$ is $G$-graded if it can be expressed as a direct sum $\Lambda = \oplus_{g \in G} \Lambda_g$ such that $\Lambda_g \cdot \Lambda_h \subseteq \Lambda_{gh}$. 
We now let $\{p_g \in kG^* \, \lvert \, g \in G \}$ be the dual basis of $kG^* :=  \Hom_k(kG, k)$, where $G$ is a finite group, and recall that 
\textit{the smash product} $\Lambda \# G^*$ of a $G$-graded algebra $\Lambda = \oplus_{g \in G} \Lambda_g$ with the group $G$ is the finite dimensional algebra with underlying vector space $\Lambda \otimes_k kG^*$ and multiplication given by 
\[
\lambda p_g \cdot \lambda' p_h = \lambda \lambda'_{g h^{-1}} p_h 
\]
wherein $\lambda' = \sum_{g \in G} \lambda'_{g}$, i.e.\ $\lambda'_{gh^{-1}}$ is the degree $gh^{-1}$ part of $\lambda'$ and we have omitted the notation for the tensor products. 
Using this, we have the following.

\begin{prop}
Let $\Lambda = \oplus_{0 \leq i \leq 2} \Lambda_i$ have the $\mathbb{Z}_2$-grading given by letting $\Lambda_0 \oplus \Lambda_2$ be in degree $0$ and $\Lambda_1$ be in degree $1$. 
Then $\Lambda^{[2]} \simeq \Lambda \# \mathbb{Z}_2^*$.  
\end{prop}
\begin{proof}
This follows by noting that
\[
\Lambda \# \mathbb{Z}_2^* \simeq 
\begin{bmatrix}
\Lambda_0p_0 & \Lambda_1p_1 \\
0 & \Lambda_0p_1
\end{bmatrix}
\oplus
\begin{bmatrix}
\Lambda_2p_0 & 0 \\
\Lambda_1p_0 & \Lambda_2 p_1
\end{bmatrix},
\]
comparing this with the decomposition given above for the $2$-quasi-Veronese, and checking that the multiplications of the two agree.
\end{proof}

The smash product $\Lambda \# G^*$ has an action of the group $G$ given by 
$(\lambda p_h)^g = \lambda p_{hg}$ for $\lambda \in \Lambda$, $p_h \in kG^*$ and $g \in G$ by \cite{Cohen-Montgomery}*{Lemma 3.3}. 
Hence, one can form the skew group algebra $(\Lambda \# G^*)G$, and this is Morita equivalent to $\Lambda$ by \cite{Cohen-Montgomery}*{Theorem 3.5}. 
Since the (\textbf{Fg}) property is preserved by Morita equivalence by the main result in \cite{Kulshammer-Psaroudakis-Skartsaeterhagen}, it follows by \cref{skew-group-transfer-prop} that if the characteristic is different from $2$, then $\Lambda$ is (\textbf{Fg}) if and only if $\Lambda^{[2]}$ is, and we record this in the following proposition.  

\begin{prop}\label{quasi-vero-transfer-prop}
Let $\Lambda = \oplus_{0 \leq i \leq 2} \Lambda_i$ have the $\mathbb{Z}_2$-grading given by letting $\Lambda_0 \oplus \Lambda_2$ be in degree $0$ and $\Lambda_1$ be in degree $1$. 
If the characteristic is different from $2$, then $\Lambda$ is \emph{(\textbf{Fg})} if and only $\Lambda^{[2]}$ is \emph{(\textbf{Fg})}.
\end{prop}

Let now $A$ be a basic hereditary $k$-algebra with quiver $Q$.
For instance, this holds for $A = \Lambda^{[2]}_0$ with our assumptions. 
Moreover, let $\Aut(Q_0;d)$ be the group of automorphisms of $Q_0$ that preserve the number of arrows between pairs of idempotents, let $\Out(A)$ be the group of outer automorphisms and let $\Out_0(A)$ be its subgroup consisting of those automorphisms that fix $Q_0$.
Recall that $\Out(A)$ is defined to be the quotient of $\Aut(A)$ by the subgroup of inner automorphisms, i.e.\ those defined by conjugation with an invertible element of $A$, and one proceeds similarly for  $\Out_0(A)$.
From \cite{Miyachi-Yekutieli}*{Proposition 1.7}, we have the following split exact sequence
\[
1 \rightarrow \Out_0(A) \rightarrow \Out(A) \rightarrow \Aut(Q_0;d) \rightarrow 1
\]
provided $A$ is connected. However, consulting the proof in  \cite{Miyachi-Yekutieli}*{Proposition 1.7}, we note that the result also holds without the connected assumption whenever $\Out_0(A)$ is trivial, which is what we need: In fact, by \cite{Miyachi-Yekutieli}*{Proposition 1.7} we have that $\Out_0(A)$ is trivial if the quiver of $A$ is a tree, i.e.\ in all Dynkin and extended Dynkin cases with the exception of $\widetilde{A}_n$. Moreover, by  \cite{Erdmann-Solberg-2}*{Proposition 2.1} we know that a self-injective radical-cube-zero algebra must have type Dynkin or extended Dynkin if it is to have (\textbf{Fg}).

From the split exact sequence above we deduce that any outer automorphism $\psi$ of $A$ is a composition consisting of an outer automorphism $\sigma$ of $A$ that fixes the idempotents corresponding to $Q_0$ and an element of $\Aut(Q_0;d)$, say $\pi$. To be more precise, $\pi$ is obtained as an element in $\Out(A)$ by mapping a preimage $\pi' \in \Aut(Q_0;d)$ into $\Aut(Q)$ -- the group of quiver automorphisms of $Q$ -- along a splitting, and one then takes the image of the resulting quiver automorphism in $\Out(A)$.

We would now like to use these observations to deduce that the Nakayama automorphisms of the algebras of interest to us are suitably nice provided their types are not $\widetilde{A}_n$.
To do this, we need a few observations:

\begin{prop}
 Let $\sigma$ and $\pi$ be automorphisms of $A$ and let $\sigma$ be inner. Then $D(A)_{\sigma \circ \pi} \simeq D(A)_{\pi}$ as $A$-bimodules.
\end{prop}
\begin{proof}
That an automorphism $\sigma$ of an algebra $A$ is inner is well-known to be equivalent to there existing an $A$-bimodule isomorphism $_{\sigma}A \simeq A$. 
The implication we need follows by noting that if $\sigma(a) = uau^{-1}$, then $a \mapsto ua$ is an isomorphism from $A$ to $_{\sigma}A$.
Twisting this by $\pi$ on the left yields an isomorphism 
$${}_{\pi}A \simeq  {}_{\pi}({}_{\sigma}A) \simeq {}_{\sigma \circ \pi}A,$$ 
and the claim then follows by dualizing. 
\end{proof}

To use this, we need the following, whose proof we omit.

\begin{prop}
Graded isomorphisms of algebras $\Gamma = \Gamma_0 \oplus \Gamma_1$ and $\Gamma' = \Gamma_0 \oplus \Gamma'_1$ whose degree $0$ parts are just the identity correspond to $\Gamma_0$-bimodule isomorphisms from $\Gamma_1$ to $\Gamma'_1$. 
\end{prop}

Note that this combined with the preceding result implies that the isomorphism class of a twisted trivial extension is determined by the outer automorphism class of the defining automorphism.

Finally, we want to transfer a Nakayama automorphism along an algebra isomorphism.

\begin{prop} 
If $\phi \colon \Lambda \rightarrow \Gamma$ is an isomorphism of algebras and $\Gamma$ is Frobenius with Nakayama automorphism $\nu$, then $\Lambda$ is Frobenius with Nakayama automorphism $\phi^{-1} \circ \nu \circ \phi$.
\end{prop}
\begin{proof}
One can note that we have the following chain of isomorphisms of $\Lambda$-bimodules
\[
\begin{tikzcd}
D(\Lambda) \arrow[r, "D(\phi^{-1})"] & D(\Gamma) \rar & \Gamma_{\nu} \rar["\phi^{-1}"] & \Lambda_{\phi^{-1} \circ \nu \circ \phi}  
\end{tikzcd}
\]
wherein $D(\Gamma)$ and $\Gamma_{\nu}$ have the $\Lambda$-bimodule structures induced by $\phi$. Indeed, the middle map exists by the assumption that $\Gamma$ is Frobenius, whereas the final map is verified by noting that $\phi^{-1}(\phi(x) x' \nu(\phi(x''))) = x\phi^{-1}(x')\phi^{-1}(\nu(\phi(x'')))$ holds for $x,x',x'' \in \Gamma$.
\end{proof}

Combining the preceding three propositions with the fact that $\Out_0(A)$ is trivial for $A$ tame hereditary of type not $\widetilde{A}_n$, we get the next result.

\begin{prop}
Let $A$ be hereditary, and let $\sigma$ and $\pi$ be automorphisms such that $\sigma \in \Out_0(A)$ and $\pi \in \Aut(Q_0;d)$. If $\sigma$ is inner, then $\Delta_{\sigma \circ \pi}A$ and $\Delta_{\pi} A$ are isomorphic as algebras. Moreover, if $A$ is tame not of type $\widetilde{A}_n$, then any twisted trivial extensions of $A$ can be endowed with a Nakayama automorphism of finite order.
\end{prop}

Finally, we are ready to show the main result of this note.

\begin{thm}
Let $\Lambda$ be Frobenius of infinite representation type satisfying that
$\rad^3 \Lambda = 0 \neq \rad^2 \Lambda$. Moreover, assume that both $2$ and the order of the Nakayama automorphism of $\Lambda$ is invertible in $k$. Then $\Lambda$ is \emph{(\textbf{Fg})} if 
\begin{enumerate}[(1)]
    \item the type of $\Lambda$ is extended Dynkin but not $\widetilde{A}_n$; or
    \item the type of $\Lambda$ is $\widetilde{A}_n$ and its Nakayama automorphism is of finite order.
\end{enumerate}
\end{thm}
\begin{proof}
Let $A = \Lambda^{[2]}$, and recall that the type of $A_0$ coincides with the type of $\Lambda$.
If the type of $\Lambda$ is extended Dynkin but not $\widetilde{A}_n$, then by the preceding result, $A$ has a Nakayama automorphism of finite order.
Hence, if either \textit{(1)} or \textit{(2)} hold, we can form the skew group algebra $AG$ with $G$ the finite cyclic group generated by a Nakayama automorphism of $A$, and $AG$ is then a radical-cube-zero symmetric algebra by \cite{Reiten-Riedtmann}*{Theorem 1.3} and \cref{skew-group-is-sym-lem}.
A basic version of $AG$ is then also radical-cube-zero symmetric by \cref{basic-version-also-sym-prop}, and these all satisfy (\textbf{Fg}) by \cite{Erdmann-Solberg-1}.
By \cite{Kulshammer-Psaroudakis-Skartsaeterhagen}, we have that $AG$ is (\textbf{Fg}) if  a basic version of $AG$ is, by \cref{skew-group-transfer-prop} we have that $A$ is (\textbf{Fg}) if  $AG$ is, and by \cref{quasi-vero-transfer-prop} we have that $\Lambda$   is (\textbf{Fg}) if $A = \Lambda^{[2]}$ is.
\end{proof}

\begin{remark}
We note some similarities with Said's approach in \cite{Said'15}:
whereas we use a $\mathbb{Z}_2$-covering or smash product via the 2-quasi-Veronese, Said employs a $\mathbb{Z}$-covering. 
The goal is more or less the same, namely getting a ``normal form'' that is easier to work with. 

Following this, we use some results on the structure of automorphisms of hereditary algebras to construct isomorphisms to further simplify the algebras in all cases except $\widetilde{A}_n$. 
Said, on the other hand, uses explicit base changes for the same end, although she can characterize when (\textbf{Fg}) holds in all of type $\widetilde{A}_n$.
Note, however, that the proofs in \cite{Said'15} are far longer and involve explicit case by case computations. 
\end{remark} 

By combining the theorem above with the main result of \cite{Said'15}, we obtain a complete answer to when a radical-cube-zero self-injective algebra satisfies the (\textbf{Fg}) property. 
Before we can present this, we must recall some things: Note that if $\Lambda$ is radical-cube-zero self-injective of type $\widetilde{A}_n$ and finite complexity, then the main result of \cite{Erdmann-Solberg-2} shows that $\Lambda$ has a quiver $Q$ as described in Proposition 5.1, Proposition 5.4 or Proposition 6.4 of \cite{Erdmann-Solberg-2}. 
Also note that \cite[Chapter 2]{Said'15} shows that $\Lambda$ is (\textbf{Fg}) in either of the latter two cases. 
If the quiver of $\Lambda$ is instead as in Proposition 5.1 of \cite{Erdmann-Solberg-2}, Said shows that one can assume that at most one commutativity relation in the presentation of $\Lambda = kQ/I$ involves a coefficient $q$ possibly different from $\pm 1$. 
Using this, Said then shows that $\Lambda$ is (\textbf{Fg}) if and only if this $q$ is a root of unity. 

\begin{cor}
Let $\Lambda = kQ/I$ be a connected Frobenius algebra satisfying that
$\rad^3 \Lambda = 0 \neq \rad^2 \Lambda$. 
Moreover, assume that both $2$ and the order of the Nakayama automorphism of $\Lambda$ is invertible in $k$. Then $\Lambda$ is \emph{(\textbf{Fg})} if and only if one of the following hold.
\begin{enumerate}[(1)]
    \item The type of $\Lambda$ is Dynkin.
    \item The type of $\Lambda$ is extended Dynkin but not $\widetilde{A}_n$.
    \item The type of $\Lambda$ is $\widetilde{A}_n$ and its quiver $Q$ is as described in Proposition 5.4 or 6.4 of \cite{Erdmann-Solberg-2}.
    \item The type of $\Lambda$ is $\widetilde{A}_n$, its quiver $Q$ is as described in Proposition 5.1 of \cite{Erdmann-Solberg-2}, and the Nakayama automorphism of $\Lambda$ is of finite order as an outer automorphism.
\end{enumerate}
\end{cor}
\begin{proof}
The ``if'' direction for \textit{(2)}-\textit{(4)} is clear by the above theorem in combination with the main result of \cite{Said'15}, whereas for \textit{(1)} it follows by Proposition 2.1 of \cite{Erdmann-Solberg-2} and the fact that representation finite selfinjective algebras are known to satisfy (\textbf{Fg}); see as before e.g.\ \cite{Dugas'10} and \cite{Green-et-al-2}.

For the ``only if'' direction, we assume that $\Lambda$ satisfies (\textbf{Fg}) and begin by considering cases: If $\Lambda$ is representation finite, the type must be Dynkin by the proof of \cite[Proposition 2.1]{Erdmann-Solberg-2}.
If $\Lambda$ is representation infinite, then the proof of \cite[Proposition 2.1]{Erdmann-Solberg-2} implies that the type must be extended Dynkin as otherwise it would have non-finite complexity in contradiction to the assumption that $\Lambda$ is (\textbf{Fg}). 
Hence, we only need to consider the case when $\Lambda$ is of type $\widetilde{A}_n$ and its quiver $Q$ is as in Proposition 5.1 of \cite{Erdmann-Solberg-2}. 

We now want to use \cite[Proposition 2.15]{Zimmermann11} in manner similar to  what is done in the proof of \cite[Criterion 5.1]{Lambre-Zhou-Zimmermann} to deduce that the Nakayama automorphism must be of finite order in this case. 
Thus, recall that we can choose a basis $\mathcal{B}$ of $\Lambda$ that also contains a basis for the socle of $\Lambda$, and let us define a function $\tr_{\Lambda} \colon \Lambda \rightarrow k$ to be given by $\tr_{\Lambda}(\lambda) = 1$ if $\lambda \in \soc \Lambda$ and $0$ otherwise. 
Then \cite[Proposition 2.15]{Zimmermann11} tells us that the bilinear form $\langle -, - \rangle$ of $\Lambda$ that exists by the assumption that $\Lambda$ is Frobenius can be given by $\langle \lambda, \lambda' \rangle = \tr_{\Lambda}(\lambda \lambda')$.

By our assumptions, the socle of $\Lambda$ consists only of paths of length two. 
Moreover, by considering the form of the quiver of $\Lambda$, we see that for any given arrow $\alpha \in Q_1$ there is exactly one arrow $\alpha^*$ such that $\alpha\alpha^{*} \in \soc \Lambda$.  
Since we have both $\rad^3 \Lambda = 0$ and that $\Lambda$ is Frobenius, paths $\lambda,\lambda'$ of length two starting and ending in the same vertex must be equal up to some scalar. 
However, by \cite[Chapter 2, Section 3]{Said'15}, we can assume any coefficient in a commutativity relation to be a root of unity and hence the same is true of any such scalar. 
Consequently, since for $\alpha \in Q$ we have $\langle \alpha, \alpha^{*} \rangle = \langle \alpha^{*}, \nu(\alpha) \rangle$, we see that the preceding observations combined with \cref{Holm-Zim-lem} imply we can choose a Nakayama autmorphism $\nu$ such that $\nu(\alpha)$ equals $q (\alpha^{*})^{*}$ for $q$ some root of unity. 
Hence, since any two Nakayama automorphisms are equal up to composition with an inner automorphism, we deduce that any Nakayama automorphism of $\Lambda$ is of finite order as an outer automorphism. This finishes the proof. 
\end{proof}

\bibliography{bib.bib}
\bibliographystyle{alpha}
\end{document}